\documentclass[preprint,12pt,authoryear]{elsarticle}

\usepackage[a4paper,verbose,
tmargin=2.5cm,bmargin=2.5cm,
lmargin=2.5cm,rmargin=2.5cm]{geometry}

\usepackage[hidelinks,colorlinks]{hyperref} 
\hypersetup{
	colorlinks=true,
	linkcolor=blue,   
	citecolor=blue,   
	urlcolor=blue     
}

\usepackage{amsmath,amssymb,amsthm} 
\usepackage{mathtools}
\numberwithin{equation}{section}    
\usepackage{algorithm}
\usepackage{algpseudocode}          
\usepackage{graphicx}               
\usepackage{booktabs}               
\usepackage{hyperref}               
\usepackage{booktabs}     
\usepackage[table]{xcolor} 
\usepackage{siunitx}      
\usepackage{caption}      

\newtheorem{theorem}{Theorem}[section]

\newtheorem{assumption}{Assumption}[section]

\DeclarePairedDelimiterX{\set}[1]{\{}{\}}{\setargs{#1}}  
\NewDocumentCommand{\setargs}{>{\SplitArgument{1}{;}}m}
{\setargsaux#1}
\NewDocumentCommand{\setargsaux}{mm}
{\IfNoValueTF{#2}{#1} {#1\,\delimsize|\,\mathopen{}#2}}

\DeclareMathOperator*{\argmin}{argmin}

\journal{Computers \& Chemical Engineering}

\begin{document}
	
	\begin{frontmatter}
		
		\title{A Sensitivity-Based Method for Bilevel Optimization Problems: Theoretical Analysis and Computational Performance}
		
		\author[inst1]{Eduardo Nolasco\corref{cor1}}
		\author[inst1,inst2]{Ross D. King}
		\author[inst3]{Vassilios S. Vassiliadis}
		
		\cortext[cor1]{Corresponding author. Email: eduardo.nolasco@cantab.net}
		
		\affiliation[inst1]{organization={Department of Chemical Engineering and Biotechnology, University of Cambridge},
			city={Cambridge}, 
			country={United Kingdom}}
		\affiliation[inst2]{organization={Department of Computer Science and Engineering, Chalmers University of Technology},
			city={Gothenburg}, 
			country={Sweden}}
		\affiliation[inst3]{organization={Independent Academic Consultant, Retired from the University of Cambridge},
			city={Larnaca}, 
			country={Cyprus}}
		
		\begin{abstract}
Bilevel optimization provides a powerful framework for modelling hierarchical decision-making systems. This work presents a sensitivity-based algorithm that addresses the bilevel structure directly by treating the lower-level optimal solution as an implicit, locally differentiable function of the upper-level variables, thereby avoiding classical single-level reformulations. Under standard regularity assumptions on the lower level, an adjoint-based representation of the reduced upper-level gradient is derived, replacing explicit construction of the sensitivity Jacobian with a single linear adjoint solve per iteration and reducing gradient evaluation cost by a factor equal to the upper-level dimension. The reduced problem is solved within an Augmented Lagrangian framework, with inner subproblems managed by an L-BFGS-B quasi-Newton solver. Convergence to KKT points of the reduced problem is established, and these points are shown to be equivalent to S-stationary solutions of the associated mathematical programme with complementarity constraints under MPEC-LICQ. Computational experiments on benchmark bilevel problems validate the method's correctness and robustness, and demonstrate the effectiveness of a pragmatic dual-criterion stopping condition in handling the asymmetric primal-dual convergence rates characteristic of augmented Lagrangian methods.
		\end{abstract}
		
		\begin{keyword}
			Bilevel Optimization \sep Sensitivity Analysis \sep Augmented Lagrangian Method \sep Adjoint Method \sep Strong Stationarity \sep Mathematical Program with Complementarity Constraints
		\end{keyword}
		
	\end{frontmatter}
	
	
	\section{Introduction}\label{Sect:Introduction}
	
	Bilevel optimization problems (BLPs) are mathematical programs in which one optimization problem is constrained by the solution of a subordinate optimization problem. This hierarchical structure defines two levels: an upper-level (or leader) problem, whose decisions influence the feasible set or objective of a lower-level (or follower) problem. The solution to the lower-level problem, in turn, feeds back into the upper-level decision-making process, creating a coupled dependency between the two levels.
	
	The origin of BLPs can be traced back to leader-follower games introduced by \cite{Stackelberg1934} within the economic context. They were later introduced to the operation research community by \cite{Bracken1973} as optimization problems with an optimization problem in their constraints, and have since found widespread application in multiple disciplines. In chemical engineering, notable examples include the optimal design of processes involving thermodynamic equilibrium \citep{Clark1990}, parameter estimation in phase equilibrium problems \citep{Bollas2009}, capacity planning \citep{Garcia-Herreros2016}, supply chain management \citep{Yue2014}, among others.
	
	Despite their practical relevance, BLPs pose significant computational challenges. Their feasible region is often discontinuous and non-differentiable, rendering the overall problem nonconvex, even when each level is convex. In addition, there exist multiple, non-equivalent formulations of BLPs, which complicate the derivation of general optimality conditions. A common approach is to reformulate the problem as a single-level optimization problem. However, such reformulations are not always faithful: specific regularity and structural assumptions must be satisfied to preserve equivalence with the original bilevel structure \citep{Dempe2015}. Classical approaches include KKT-based reformulations leading to mathematical programs with complementarity constraints, exact global methods based on bounding schemes, multiparametric programming for problems with linear or quadratic lower levels, and data-driven surrogate approaches; each carries significant limitations in the general smooth nonlinear setting considered here, as reviewed in Section~\ref{Sect:Formulation}.
	
	This work proposes a sensitivity-based descent algorithm for solving deterministic, continuous bilevel optimization problems in which the lower-level problem is convex. The method leverages parametric sensitivity analysis to treat the lower-level optimal solution as an implicit, locally differentiable function of the upper-level variables, enabling the construction of descent directions for the upper-level objective without recourse to classical single-level reformulations. Gradient evaluation is performed via an adjoint system that avoids explicit formation of the sensitivity Jacobian, reducing the per-iteration gradient cost to a single linear solve independent of the upper-level dimension. The reduced problem is solved within a robust Augmented Lagrangian framework, where inner subproblems are managed by an L-BFGS-B quasi-Newton solver with a strong Wolfe line search. Convergence to KKT points of the reduced problem is established under standard regularity assumptions, and these points are shown to be equivalent to S-stationary solutions of the associated MPCC reformulation under MPEC-LICQ. A pragmatic dual-criterion stopping condition is introduced to address the asymmetric primal-dual convergence rates characteristic of augmented Lagrangian methods.
	
	The remainder of the paper is structured as follows. Section \ref{Sect:Formulation} discusses alternative formulations of BLPs and reviews classical solution approaches, including KKT-based reformulations, value-function methods, exact global bounding schemes, multiparametric programming, and data-driven surrogate approaches. Section \ref{Sect:Method} introduces the proposed sensitivity-based method, including the adjoint-based gradient formulation, the Augmented Lagrangian algorithm, and the convergence and S-stationarity equivalence results. In Section \ref{Sect:Tests} presents the computational experiments, covering implementation details, illustrative examples, and benchmark results on the BOLIB library. Section \ref{Sect:Conclusions} concludes the paper and outlines directions for future research.
	
	\section{BLP formulation and classical solution approaches}\label{Sect:Formulation}
	\subsection{Bilevel optimization problems formulation}
	The general, yet inherently ambiguous, formulation of a BLP can be expressed as:
	\begin{equation}\label{eq:ambiguous}
		\begin{aligned}
			\min_{x \in X} & \ F(x,y) \\
			\text{s.t.} & \ G(x,y) \leq 0, \\
			& \ y \in \argmin_{y \in Y} \set{ f(x,y) ; g(x,y) \leq 0 },
		\end{aligned}
	\end{equation}
	where $x \in X \subset \mathbb{R}^n$ denotes the upper-level variables, $y \in Y \subset \mathbb{R}^m$ denotes the lower-level variables. The functions $F, f:\mathbb{R}^n \times \mathbb{R}^m \to \mathbb{R}$ define the upper- and lower-level objectives, while $G:\mathbb{R}^n \times \mathbb{R}^m \to \mathbb{R}^r$ and $g:\mathbb{R}^n \times \mathbb{R}^m \to \mathbb{R}^s$ denote the vector-valued inequality constraints functions at the upper and lower-level, respectively. To simplify the notation and exposition, equality constraints are omitted in this work noting that they can be represented by pairs of inequalities. Throughout this work, uppercase letters denote upper-level functions, while lowercase letters denote lower-level functions.
	
	Despite its apparent simplicity, the BLP \eqref{eq:ambiguous} is not well-posed whenever the lower-level problem admits multiple optimal solutions. Consider the illustrative example due to \cite{Lucchetti1987} whose upper-level objective is $F(x,y) = x^2 + y^2$ and lower-level problem is $\argmin_{y} \set{-xy; 0 \leq y \leq 1}$; if the upper-level decision is $x=0$, then any $y \in [0,1]$ is optimal for the lower-level, thus rendering the BLP ill-defined.
	
	To appropriately formulate the BLP, we first introduce the \emph{parametric} lower-level problem:
	\begin{equation}\label{eq:ll}
		\begin{aligned}
			\min_{y \in Y} & \ f(x,y) \\
			\text{s.t.} & \ g(x,y) \leq 0,
		\end{aligned}
	\end{equation}
	for a given value of the upper-level variable $x$. The solution set of this problem defines a set-valued mapping $\Psi: \mathbb{R}^n \rightrightarrows \mathbb{R}^m$ given by
	\begin{equation}\label{eq:Psi(x)}
		\Psi(x) := \argmin_{y \in Y} \set{ f(x,y) ;  g(x,y) \leq 0 }.
	\end{equation}
	
	If the upper-level can influence the lower-level's decision when multiple solutions exist, then one obtains the \textit{optimistic} bilevel optimization problem:
	\begin{equation}\label{eq:op_BLP}
		\begin{aligned}
			\min_{x, y} \  & \ F(x,y) \\
			\text{s.t.} & \ G(x,y) \leq 0, \\
			& \ x \in X, \\
			& \ y \in \Psi(x).
		\end{aligned}
	\end{equation}
	It is important to note that the operator $\min_{x, y}$ in this context does not imply a simultaneous optimization over the variables $x$ and $y$. Instead, the constraint $y \in \Psi(x)$ dictates a sequential process where, for a given $x$, the upper-level selects a specific $y$ from the set $\Psi(x)$ that is most favorable to its own objective. The sets $X$ and $Y$ are typically compact sets defined by box constraints on the variables which are treated as inequality constraints and included in $G$ and $g$.
	
	If the leader cannot influence the lower-level decision making process, further assumptions or hierarchical selection rules must be imposed leading to the \textit{pessimistic} formulation of BLPs \citep{Dempe2002, Wiesemann2013}. Throughout this work, we consider the optimistic case and assume that $\Psi(x)$ is single-valued, ensuring that the BLP \eqref{eq:op_BLP} is well defined.
	
	Note that, by definition, the mapping $\Psi$ in \eqref{eq:Psi(x)} denotes the set of \emph{global minimizers} of the parametric lower-level problem \eqref{eq:ll}. Thus, any algorithm for solving bilevel problems must ensure global optimality at the lower-level. If local minimizers or stationary points of the lower-level problem are admitted, then the solution of the resulting relaxed problem will generally differ from that of the original problem \citep{Mirrlees1991}.
	
	\subsection{Classical solution methods}
	Bilevel optimization problems are intrinsically difficult to analyze and solve. In particular, optimality conditions based on classical nonlinear programming concepts (stationarity, constraints qualifications or duality) are not readily available for the bilevel case. Therefore, the usual approach to solve the BLP \eqref{eq:op_BLP} is to reformulate as a single-level optimization problem.
	
	One such reformulations replaces the lower-level problem \eqref{eq:ll} with its Karush-Kuhn-Tucker (KKT) optimality conditions, which are then included as constraints in the upper-level problem. If the functions $f, g $ are differentiable and convex, and if a suitable constraint qualification holds for all $x$ at $y \in \Psi(x)$, then the bilevel problem \eqref{eq:op_BLP} can be reformulated as its \emph{KKT transformation}:
	\begin{equation}\label{eq:BLP MPCC}
		\begin{aligned}
			\min_{x, y, \lambda} & \ F(x,y) \\
			\text{s.t.} & \ G(x,y) \leq 0 \\
			& \ \nabla_y f(x,y) + \lambda^\top \nabla_y g(x,y) =0, \\
			& \ g(x,y) \leq 0 \\
			& \ 0 \leq \lambda, \ \lambda^\top g(x,y)=0 \\
			& \ y \in Y, \ x \in X.
		\end{aligned}
	\end{equation}
	
	Problem \eqref{eq:BLP MPCC} is a \emph{mathematical program with complementarity constraints (MPCC)}. Both bilevel optimization problems and MPCCs are special cases of the broader class of \emph{mathematical programs with equilibrium constraints (MPECs)} \citep{Kocvara2004}. Due to the complementarity constraint,  MPCCs violate standard constraint qualifications at any feasible point, which makes the derivation of optimality conditions a challenging task. To address these challenges, several generalized \emph{stationarity concepts} have been developed within the MPEC framework \citep{Outrata1990, Scheel2000}. The strongest among these is \emph{Strong stationarity (S-stationarity)}. A feasible point of the MPCC reformulation \eqref{eq:BLP MPCC} is called S-stationary if its standard KKT conditions are satisfied. This means there exist Lagrange multipliers such that the gradient of the MPEC Lagrangian is zero, and the multipliers associated with the inequality constraints are all non-negative. An S-stationary point is a highly desirable solution, as it is the most rigorous of the MPCC stationarity conditions. This MPCC reformulation has been studied in \cite{Gumus2001}.
	
	Another classical reformulation of the bilevel problem \eqref{eq:op_BLP} is the \emph{optimal value function} reformulation, originally introduced by \cite{Outrata1990}. In this formulation, the optimal value function associated with the lower-level problem \eqref{eq:ll} is defined as
	\begin{equation}
		\varphi(x) = \min_{y \in Y} \set{f(x,y) ; g(x,y) \leq 0},
	\end{equation}
	and is subsequently introduced as a constraint in the upper-level problem. By combining this with the feasible set of the lower-level problem, the bilevel problem is reformulated as the following single-level problem:
	\begin{equation}\label{eq:optimal value BLP}
		\begin{aligned}
			\min_{x, y} \  & \ F(x,y) \\
			\text{s.t.} \ & \ G(x) \leq 0 \\
			& \ f(x,y) \leq \varphi(x) \\
			& \ g(x,y) \leq 0 \\
			& \ y \in T \\
			& \ x \in X.
		\end{aligned}
	\end{equation}
	Problem \eqref{eq:optimal value BLP} is fully equivalent to the original bilevel problem \eqref{eq:op_BLP} both in terms of local and global solutions \citep{Dempe2015}. However, direct solution of \eqref{eq:optimal value BLP} is computationally intractable in general, since $\varphi(x)$ is implicit, nonsmooth, and nonconvex, and no closed-form expression is available. This has motivated two distinct strategies in the global optimization literature, both of which use $\varphi(x)$ as a conceptual object without computing it explicitly.
	
	The first strategy preserves the bilevel structure and constructs convergent bounds on $\varphi(x)$ through nested bounding problems. \cite{Kleniati2014} proposed a branch-and-sandwich algorithm that maintains separate upper- and lower-level bounding problems, sandwiching $\varphi(x)$ from above and below without collapsing the problem to the single-level form \eqref{eq:optimal value BLP}. This was later extended to mixed-integer bilevel problems \citep{Kleniati2015}.
	
	The second strategy avoids explicit reference to $\varphi(x)$ altogether, instead constructing convergent approximations of the lower-level optimality constraint through discretisation-based bounding schemes. \cite{Mitsos2008} introduced a bounding algorithm for the global solution of continuous bilevel programmes with nonconvex lower levels. The approach constructs convergent lower and upper bounds by solving a relaxed single-level programme augmented with parametric upper bounds on the lower-level optimal value, guaranteeing $\varepsilon$-optimality without branching. This framework was subsequently extended to handle mixed-integer variables \citep{Mitsos2010}, and later adapted to accommodate lower-level equality constraints \citep{DjelassiMitsos2019}, which arise naturally in process systems engineering applications where the lower-level problem encodes equilibrium conditions or steady-state process models.
	
	Another approach to solve the optimistic bilvel problem \eqref{eq:op_BLP}, developed within the process systems engineering community, exploits the structure of bilevel problems with linear or quadratic lower levels through multiparametric programming. \cite{Faisca2007} first proposed solving the lower-level problem as a multiparametric programme parametrized by the upper-level variables, reducing the bilevel problem to a sequence of independent LP or QP problems solvable to global optimality. This was subsequently extended to mixed-integer linear and quadratic lower levels by \cite{Avraamidou2019}, providing exact global solutions for B-MILP and B-MIQP problems. A general reference for multiparametric programming in the multilevel case is found in \cite{Avraamidou2022}.
	
	In contrast to exact methods, data-driven approaches sacrifice optimality guarantees in exchange for the ability to handle general nonlinear and black-box problem structures. \cite{Beykal2020} proposed DOMINO, a framework that reformulates the bilevel problem as a single-level grey-box optimization problem by sampling the upper-level objective at points where the lower-level has been solved to global optimality, and applying a derivative-free solver to the resulting surrogate. DOMINO provides bilevel feasibility guarantees and near-optimal solutions across a range of problem classes including B-MINLP, but does not provide a certificate of global optimality, as the surrogate approximation of the upper-level objective does not in general reproduce the true bilevel structure. Surrogate-assisted evolutionary approaches \citep{Islam2017, Sinha2022} follow a similar paradigm, constructing Kriging or other regression models of the lower-level response and embedding these within population-based search; these methods share the same limitation that solution quality depends on surrogate accuracy, which degrades in high dimensions.
	
	The single-level reformulations and solution strategies reviewed above each carry significant limitations for the general smooth nonlinear bilevel setting considered here. The KKT reformulation \eqref{eq:BLP MPCC} produces an MPCC that violates standard constraint qualifications at every feasible point, complicating both optimality theory and algorithmic design. The value-function reformulation \eqref{eq:optimal value BLP}, while theoretically equivalent to the original problem, is computationally intractable: $\phi(x)$ is implicit, nonsmooth, and nonconvex, with no closed-form expression available. Exact global methods that work around these difficulties — whether through bounding schemes or parametric solution maps — either require certified global optimality of the lower-level problem at every iteration, incurring substantial overhead, or are restricted to lower levels with linear or quadratic structure. Data-driven and surrogate-based approaches handle general nonlinear structures but approximate the bilevel objective rather than operating on it directly, providing no certificate on the gap to the true optimum.
	
	The method proposed in this work takes a different route. Under certain regularity conditions (see next section), the lower-level solution $\bar{y}(x)$ is a continuously differentiable implicit function of $x$, and its sensitivity with respect to $x$ can be computed efficiently via an adjoint system. This allows the bilevel problem to be solved directly as a smooth optimization problem in $x$ alone, without constructing any single-level reformulation, while retaining the theoretical guarantees of KKT-based optimality conditions as developed in the following section.

\section{The Sensitivity-Based Solution Method}\label{Sect:Method}
This section details the proposed sensitivity-based algorithm for solving the optimistic bilevel problem \eqref{eq:op_BLP}. Our method belongs to the class of gradient-based algorithms and follows a nested approach where an outer loop updates the upper-level variables and an inner loop solves the lower-level problem for a given upper-level decision. The main idea is to circumvent the single-level reformulations discussed previously by treating the lower-level problem as a parametric optimization problem and the upper-level as an implicit problem.

\subsection{The Implicit Upper-Level Problem}
The algorithm is based on the insight that the lower-level's optimal decision can be viewed as an \emph{implicit function} of $x$, denoted $\bar{y}(x)$. This allows to transform the bilevel problem \eqref{eq:op_BLP} into an equivalent single-level problem:
\begin{equation}\label{eq:Implicit_BLP}
	\begin{aligned}
		\min_{x \in X} \quad & F(x, \bar{y}(x))\\
		\text{s.t.} \quad & G(x, \bar{y}(x)) \le 0.
	\end{aligned}
\end{equation}

While this problem cannot be solved directly, as $\bar{y}(x)$ is not known in closed form, this formulation enables a gradient-based solution strategy. The total derivatives of the functions in the above problem can be computed via sensitivity analysis, as detailed in the subsequent sections. 

For this transformation to be valid and for the implicit function $\bar{y}(x)$ to be locally unique and continuously differentiable, we impose the following standard assumptions.

\begin{assumption}[Regularity Conditions]\label{ass:main}
	Let $(\bar{x}, \bar{y})$ be a feasible point of the bilevel problem \eqref{eq:op_BLP}.
	\begin{enumerate}
		\item[(a)] \textbf{Smoothness:} The functions $F, G, f,$ and $g$ are twice continuously differentiable in a neighborhood of $(\bar{x}, \bar{y})$.
		\item[(b)] \textbf{Lower-Level Convexity:} For any feasible $x$ in the neighborhood of $\bar{x}$, the lower-level problem \eqref{eq:ll} is strictly convex.
		\item[(c)] \textbf{Lower-Level Regularity:} For any feasible $x$ in the neighborhood of $\bar{x}$, the lower-level solution $\bar{y}(x)$ satisfies the Linear Independence Constraint Qualification (LICQ), the Second-Order Sufficient Condition (SOSC), and the Strict Complementarity Condition (SCC).
	\end{enumerate}
\end{assumption}

The above assumptions ensure, via the Implicit Function Theorem, that the solution map $\bar{y}(x)$ and its associated Lagrange multipliers $\bar{\lambda}(x)$ are continuously differentiable functions in the neighborhood of $\bar{x}$. This differentiability is fundamental to compute derivatives of the upper-level problem \eqref{eq:Implicit_BLP}. For notational clarity throughout the remainder of this work, we omit the explicit dependency of the optimal lower-level solution $\bar{y}$ and its associated multipliers $\bar{\lambda}$ on $x$.

The strict convexity assumption in Assumption \ref{ass:main}(b) excludes the important class of bilevel optimization problems whose lower-level is a linear program (LP). To accommodate this class of problems within our framework, we employ a standard regularization technique. The linear objective $f(x,y) = c(x)^T y$ is replaced by the regularised objective $f_{\varepsilon}(x,y) = c(x)^T y + \varepsilon \|y\|^2$, where $\varepsilon$ is a small, positive constant (e.g., $10^{-6}$). This ensures the LP has a unique solution and satisfies the necessary regularity conditions for sensitivity analysis, allowing the rest of the algorithm to be applied without modification.

\subsection{Sensitivity Analysis of the Lower Level}
The optimality of the lower-level problem \eqref{eq:ll} for a fixed $x$ is characterized by its Karush-Kuhn-Tucker (KKT) conditions. These are derived from the problem's Lagrangian function:
\begin{equation}\label{eq:ll_lagrangian}
	\mathcal{L}_f(x, y, \lambda) = f(x, y) + \lambda^{\top} g(x, y),
\end{equation}
where $\lambda \in \mathbb{R}^s$ are the Lagrange multipliers. The KKT conditions must hold at the optimal solution $(\bar{y}, \bar{\lambda})$. The lower-level KKT conditions are:
\begin{subequations}\label{eq:ll_kkt}
	\begin{align}
		\nabla_{y} \mathcal{L}_f(x, \bar{y}, \bar{\lambda}) &= 0, \label{eq:KKT_stationarity} \\
		g(x, \bar{y}) &\le 0, \label{eq:KKT_feasibility}\\
		\bar{\lambda} &\ge 0, \label{eq:KKT_dual}\\
		\bar{\lambda}_i g_i(x, \bar{y}) &= 0, \quad \forall i=1, \dots, s. \label{eq:KKT_complementarity}
	\end{align}
\end{subequations}

Because of the complementarity constraints \eqref{eq:KKT_complementarity}, inactive constraints do not play a role in the optimization process. For sensitivity analysis, we only consider the constraints that are binding at the solution. We define the \emph{active set} at the solution $\bar{y}$ as \mbox{$A(x,\bar{y}) = \set{i; g_i(x,\bar{y}) = 0}$}.

If the conditions in Assumption \ref{ass:main} are satisfied, we can differentiate the stationarity condition \eqref{eq:KKT_stationarity} and the complementarity constraints \eqref{eq:KKT_complementarity} for the active constraints. This yields the following linear system for the sensitivities:
\begin{equation} \label{eq:sensitivity_system}
	\underbrace{\begin{bmatrix}
			H_{\mathcal{L}_f} & \nabla_{y} g_A^{\top} \\
			\Lambda_A \nabla_{y} g_A & 0
	\end{bmatrix}}_{=:\,M}
	\begin{bmatrix}
		\dfrac{d\bar{y}}{dx} \\[6pt]
		\dfrac{d\bar{\lambda}_A}{dx}
	\end{bmatrix}
	=
	\begin{bmatrix}
		- \nabla_{yx}^2 \mathcal{L}_f \\
		- \Lambda_A \nabla_{x} g_A
	\end{bmatrix},
\end{equation}
where $H_{\mathcal{L}_f}$ is the Hessian of the Lagrangian \eqref{eq:ll_lagrangian} with respect to $y$ evaluated at $(x, \bar{y}, \bar{\lambda})$; $g_A$ is the vector of active constraints, $g_i$ for $i \in A(x,\bar{y})$; and $\Lambda_A = \text{diag}(\bar{\lambda}_i)_{i \in A}$ is the diagonal matrix of active multipliers. Under Assumption~\ref{ass:main}, the matrix $M$ is nonsingular, so the Jacobian $J := d\bar{y}/dx$ is uniquely defined by \eqref{eq:sensitivity_system}.

\subsection{Total Gradient Computation for the Upper Level}
To solve the implicit upper-level problem \eqref{eq:Implicit_BLP} via a gradient-based method, we need the total derivatives of its objective and constraint functions with respect to $x$. Using the sensitivity term $d\bar{y}/dx$ from the linear system \eqref{eq:sensitivity_system}, these gradients are computed via the chain rule:
\begin{subequations}\label{eq:UL_gradients}
	\begin{align}
		\nabla_{x} F(x, \bar{y}) &= \frac{\partial F}{\partial x} + \frac{\partial F}{\partial y} \frac{d \bar{y}}{dx} \label{eq:total_grad_F} \\
		\nabla_{x} G(x, \bar{y}) &= \frac{\partial G}{\partial x} + \frac{\partial G}{\partial y} \frac{d \bar{y}}{dx} \label{eq:total_grad_G}
	\end{align}
\end{subequations}
These gradients are essential for the iterative optimization procedure, which is designed to find a point satisfying the problem's KKT conditions. The Lagrangian for the upper-level problem is:
\begin{equation}\label{eq:up_lagrangian}
	\mathcal{L}_{F} (x, \mu) = F(x, \bar{y}) + \mu^{\top} G(x, \bar{y}),
\end{equation}
where $\mu \in \mathbb{R}^r$ are the Lagrange multipliers associated with the upper-level constraints. The goal of the upper-level solver is to find a point $(\bar{x}, \bar{\mu})$ that satisfies the following KKT conditions:
\begin{subequations}\label{eq:ul_KKT}
	\begin{align}
		\nabla_{x} \mathcal{L}_F (\bar{x},\bar{\mu}) &=0,\label{eq:UL_KKT_stationarity} \\
		G(\bar{x}, \bar{y}) &\leq 0, \\
		\bar{\mu} &\geq 0, \\
		\bar{\mu}_i G_i(\bar{x}, \bar{y}) &= 0, \quad \forall i=1, \dots, r.
	\end{align}
\end{subequations}

Computing the full Jacobian $J = d\bar{y}/dx$ explicitly by solving \eqref{eq:sensitivity_system} for each of the $n$ columns requires $n$ linear solves with $M$. For gradient-based upper-level methods, however, what is required is not $J$ itself but the product $J^\top q$, where $q$ encodes the upper-level multiplier information. This product can be obtained at the cost of a \emph{single} linear solve by introducing an adjoint formulation.

Define the adjoint vector
\begin{equation}\label{eq:adjoint_rhs}
	q(x,\mu) := \nabla_y F(x, \bar{y}) + \sum_{i=1}^{r} \mu_i \,\nabla_y G_i(x, \bar{y}),
\end{equation}
which collects the $y$-partial derivatives of the upper-level Lagrangian \eqref{eq:up_lagrangian}. The adjoint variables $(\nu, w)$ are then defined as the solution of the linear system
\begin{equation}\label{eq:adjoint_system}
	M^{\top} \begin{bmatrix} \nu \\ w \end{bmatrix} = -\begin{bmatrix} q(x,\mu) \\ 0 \end{bmatrix}.
\end{equation}
Since $M$ is nonsingular under Assumption~\ref{ass:main}, the adjoint system \eqref{eq:adjoint_system} has a unique solution $(\nu, w)$. Transposing \eqref{eq:sensitivity_system} and multiplying on the left by $[\nu^\top, w^\top]$, then substituting \eqref{eq:adjoint_system}, yields the identity
\[
J^\top q(x,\mu) = \left(\nabla_{yx}^2 \mathcal{L}_f\right)^\top \nu + \left(\nabla_x g_A\right)^\top \Lambda_A w.
\]
Substituting into the stationarity condition \eqref{eq:UL_KKT_stationarity}, the gradient of the upper-level Lagrangian can be written entirely in terms of the adjoint variables:
\begin{equation}\label{eq:adjoint_gradient}
	\nabla_x \mathcal{L}_F(x,\mu)
	= \nabla_x F + \sum_{i=1}^{r} \mu_i \nabla_x G_i
	+ \left(\nabla_{yx}^2 \mathcal{L}_f\right)^\top \nu + \left(\nabla_x g_A\right)^\top \Lambda_A w.
\end{equation}
Equation \eqref{eq:adjoint_gradient} requires only a single solve of \eqref{eq:adjoint_system} per gradient evaluation and avoids forming $J$ explicitly. This is the formulation used in the implementation.

The computational cost per inner iteration comprises of one lower-level NLP solve of dimension $m$ and one solve of the adjoint linear system \eqref{eq:adjoint_system} of dimension $(m + \vert A \vert)$, where $\vert A \vert \leq s$ is the number of active lower-level constraints. Forming the sensitivity Jacobian $J$ would require $n$ solves with the same matrix $M$. The adjoint formulation reduces this to a single solve regardless of the upper-level dimension $n$, so the gradient evaluation cost scales with the lower-level problem size only.

\subsection{The Proposed Algorithm}

Throughout this section we distinguish two iteration indices. The outer
index $k$ tracks the Augmented Lagrangian updates of the upper-level
problem, producing the sequence of outer iterates $x_k$ with associated
multipliers $(\mu_k, \rho_k)$. The inner index $j$ tracks the descent
iterations used to approximately minimise each fixed Augmented Lagrangian
subproblem; $x_{k,j}$ denotes the $j$-th inner iterate for fixed
$(\mu_k, \rho_k)$, and $\bar{y}_{k,j}$ is the corresponding lower-level
solution obtained by solving \eqref{eq:ll} at $x_{k,j}$.

The upper-level update step is performed using an \emph{Augmented
	Lagrangian Method (ALM)} based on the
\emph{Powell--Hestenes--Rockafellar (PHR)} augmented Lagrangian function
\citep{Powell1969, Hestenes1969, Rockafellar1973}:
\begin{equation}\label{eq:aug_Lagrangian}
	\mathcal{L}_{\rho}(x,\mu;\rho)
	= F(x,\bar{y})
	+ \frac{1}{2\rho}\sum_{i=1}^{r}
	\Bigl[\bigl(\max\{0,\,\mu_i + \rho G_i(x,\bar{y})\}\bigr)^2
	- \mu_i^2\Bigr],
\end{equation}
where $\rho > 0$ is the penalty parameter. Defining
$\hat{\mu}_i := \max\{0,\,\mu_i + \rho G_i(x,\bar{y})\}$, its first
derivative with respect to $x$ is
\begin{equation}\label{eq:diff_aug_lagr}
	\nabla_x \mathcal{L}_{\rho}(x,\mu;\rho)
	= \nabla_x F(x,\bar{y})
	+ \sum_{i=1}^{r} \hat{\mu}_i\,\nabla_x G_i(x,\bar{y}),
\end{equation}
which is a continuous function. Using the adjoint formulation
\eqref{eq:adjoint_system}--\eqref{eq:adjoint_gradient} with $\mu$
replaced by $\hat{\mu}$, this gradient is evaluated in practice without
forming the Jacobian $J$, by solving a single adjoint system per gradient
evaluation.

At each outer iteration $k$, the ALM forms the augmented Lagrangian
subproblem
\begin{equation}\label{eq:subproblem}
	\min_{x}\;\mathcal{L}_{\rho}(x,\mu_k;\rho_k).
\end{equation}

This subproblem is solved inexactly using a gradient-based method:
\begin{equation}\label{eq:x_update}
	x_{new} \gets x_{current} + \alpha p,
\end{equation}
where $p$ is the L-BFGS-B descent direction and $\alpha$ is the step
size determined via a line search satisfying the strong Wolfe conditions
\begin{subequations}\label{eq:strong_wolfe}
	\begin{align}
		\mathcal{L}_{\rho}(x+\alpha p,\mu_k;\rho_k)
		&\leq \mathcal{L}_{\rho}(x,\mu_k;\rho_k)
		+ c_1\alpha\,\nabla_x\mathcal{L}_{\rho}(x,\mu_k;\rho_k)^{\top}p,
		\label{eq:wolfe_armijo}\\
		\bigl|\nabla_x\mathcal{L}_{\rho}(x+\alpha p,\mu_k;\rho_k)^{\top}p\bigr|
		&\leq c_2\bigl|\nabla_x\mathcal{L}_{\rho}(x,\mu_k;\rho_k)^{\top}p\bigr|,
		\label{eq:wolfe_curvature}
	\end{align}
\end{subequations}
with $0 < c_1 < c_2 < 1$.

Following the primal update, the dual variables are updated using the
new iterate $x_{k+1}$:
\begin{equation}\label{eq:mu_update}
	\mu_{k+1,i} \gets \max\{0,\;\mu_{k,i} + \rho_k G_i(x_{k+1},\bar{y}_{k+1})\}.
\end{equation}
This step uses the constraint violation at the new point
$(x_{k+1}, \bar{y}_{k+1})$ to update the multipliers effectively.
Finally, the penalty parameter $\rho$ is managed by an adaptive scheme
that balances the minimisation of the objective and the satisfaction of
the constraints: the penalty is increased by a factor $\gamma > 1$ only
when the improvement in primal feasibility between iterations is deemed
insufficient, avoiding unnecessarily large penalty values that could lead
to ill-conditioning.

It is worth noting that the sensitivity and adjoint systems
\eqref{eq:sensitivity_system} and \eqref{eq:adjoint_system} are
reconstructed at each inner iterate $x_{k,j}$ using the active set
$A(x_{k,j},\bar{y}_{k,j})$ determined by the lower-level
solve at that point (Algorithm~\ref{alg:ALM_revised}, Step~9).
Consequently, the gradient $\nabla_{x}\mathcal{L}_{\rho}$ used at
iteration $j+1$ is always consistent with the active set at $x_{k,j+1}$,
and no stale sensitivity information is carried forward. If a change in
the active set occurs between iterates, the updated active set is
automatically incorporated at the next gradient evaluation. The
theoretical guarantees presented in the next subsection apply locally on
neighborhoods of active-set stability; across transitions, the algorithm
continues to produce descent steps based on the current active set,
though convergence results do not extend globally across such changes.

We use the residuals of the upper-level KKT conditions
\eqref{eq:ul_KKT} to define a stopping criterion for our algorithm. At
the end of each iteration, using the new iterate $(x_{k+1}, \mu_{k+1})$,
we define the residuals as:
\begin{subequations}\label{eq:residuals}
	\begin{align}
		r_{stat} &= \Vert \nabla_x \mathcal{L}_{F} (x_{k+1}, \mu_{k+1}) \Vert_{\infty}, \label{eq:res_stat}\\
		r_{feas} &= \Vert \max\{0, G(x_{k+1}, \bar{y}_{k+1})\} \Vert_{\infty}, \label{eq:res_feas}\\
		r_{comp} &= \Vert \text{diag}(\mu_{k+1}) G(x_{k+1}, \bar{y}_{k+1}) \Vert_{\infty}, \label{eq:res_comp}
	\end{align}
\end{subequations}
and the overall KKT residual is
\begin{equation}\label{eq:res_KKT}
	r_{KKT} = \max\{r_{stat}, r_{feas}, r_{comp}\}.
\end{equation}
The algorithm terminates when the KKT residual $r_{KKT}$ falls below the
prescribed tolerance $\epsilon > 0$. In practice, a secondary stall
criterion is also imposed: the algorithm terminates if both
$\Vert x_{k+1} - x_k \Vert_\infty < \epsilon_{stall}$ and
$\vert F(x_{k+1}, \bar{y}_{k+1}) - F(x_k,\bar{y}_k) \vert < \epsilon_{stall}$
for a prescribed tolerance $\epsilon_{stall} > 0$, detecting cases where
the iterates and the upper-level objective have ceased to make meaningful
progress.

The complete procedure is presented in Algorithm~\ref{alg:MainAlg}; a
high-level schematic of the overall framework is provided in
Figure~\ref{fig:SBMBLP}.

\begin{algorithm}
	\caption{Sensitivity-Based Algorithm for Bilevel Optimization}
	\label{alg:MainAlg}
	\begin{algorithmic}[1]
		\State \textbf{Initialize:} Set $(x_0, \mu_0, \rho_0)$ and
		tolerances $\epsilon > 0$, $\epsilon_{stall} > 0$.
		\State \textbf{Initial Solve:} Solve \eqref{eq:ll} at $x_0$ to
		obtain the KKT pair $(\bar{y}_0, \bar{\lambda}_0)$.
		\State Compute
		$r_{feas,0} = \|\max\{0, G(x_0,\bar{y}_0)\}\|_{\infty}$.
		\State Set $k \gets 0$.
		\Repeat
		\State Find an improved iterate
		$(x_{k+1}, \bar{y}_{k+1}, \bar{\lambda}_{k+1},
		\mu_{k+1}, \rho_{k+1})$
		by solving the augmented Lagrangian subproblem \eqref{eq:subproblem} using
		Algorithm~\ref{alg:ALM_revised}.\label{alg:UL_step}
		\State Calculate the overall KKT residual $r_{KKT}$ using
		\eqref{eq:residuals} at $(x_{k+1}, \mu_{k+1})$.
		\State $k \gets k+1$.
		\Until{$r_{KKT} < \epsilon$
			\textbf{ or }
			$\|x_k - x_{k-1}\|_{\infty} < \epsilon_{stall}$
			\textbf{ and }
			$|F(x_k,\bar{y}_k) - F(x_{k-1},\bar{y}_{k-1})| <
			\epsilon_{stall}$}\\
		\Return Final solution $(\bar{x}, \bar{y})$.
	\end{algorithmic}
\end{algorithm}

\begin{figure}[t]
	\centering
	\includegraphics[width=0.95\linewidth]{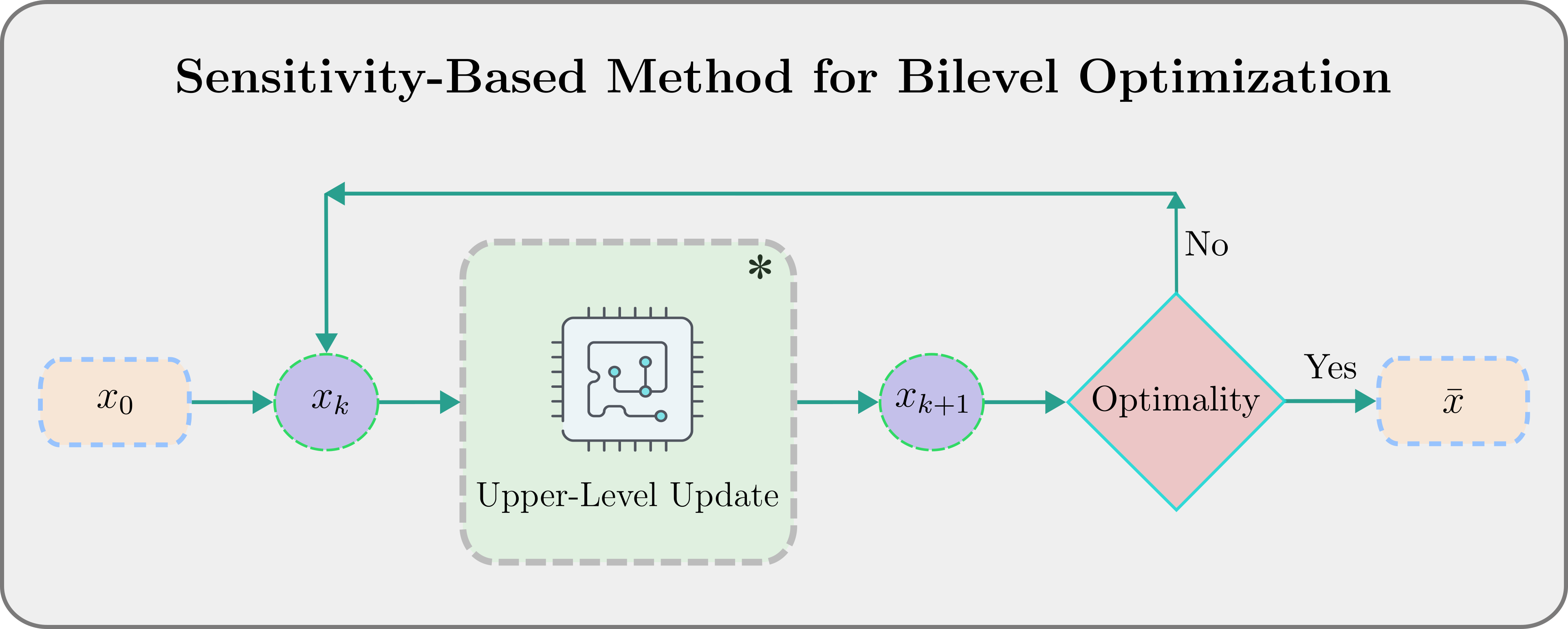}
	\caption{High-level schematic of the overall sensitivity-based
		Augmented Lagrangian framework, corresponding to
		Algorithm~\ref{alg:MainAlg}.}
	\label{fig:SBMBLP}
\end{figure}

The upper-level update step (Step~\ref{alg:UL_step} in
Algorithm~\ref{alg:MainAlg}) is implemented as
Algorithm~\ref{alg:ALM_revised}. The most computationally intensive
part of each inner iteration is the evaluation of the gradient
$\nabla_x \mathcal{L}_\rho$, which is an implicit function of $x$; the
multi-step workflow, involving solving the lower-level NLP and the
adjoint sensitivity system, is illustrated in
Figures~\ref{fig:UL_update}--\ref{fig:grad_computation}.

\begin{figure}[t]
	\centering
	\includegraphics[width=0.95\linewidth]{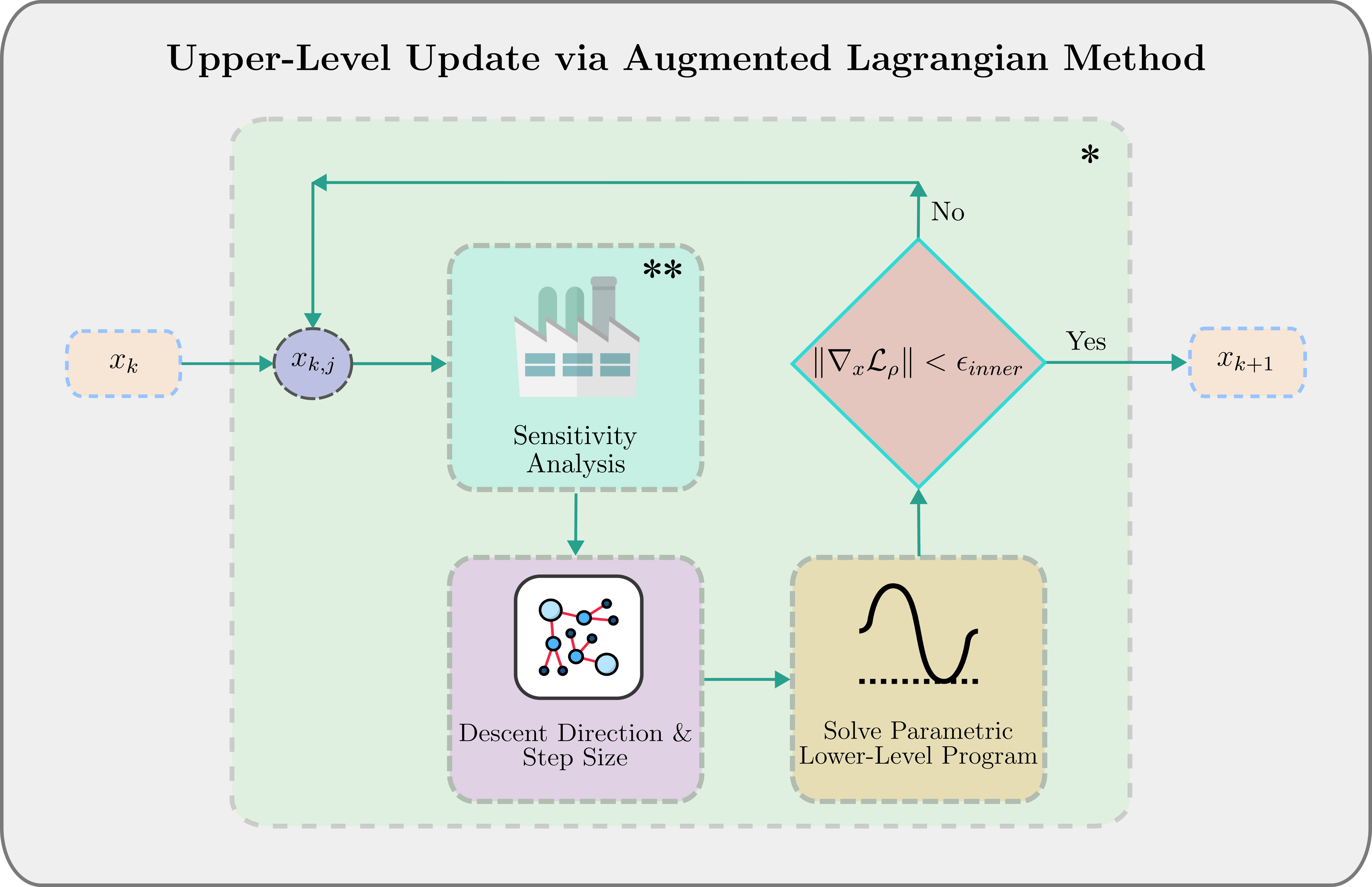}
	\caption{Workflow for the ALM subproblem solution
		(Algorithm~\ref{alg:ALM_revised}). At each outer iteration,
		a NLP solver is used to find an approximate minimiser of the
		implicit Augmented Lagrangian function.}
	\label{fig:UL_update}
\end{figure}

\begin{figure}[t]
	\centering
	\includegraphics[width=0.95\linewidth]{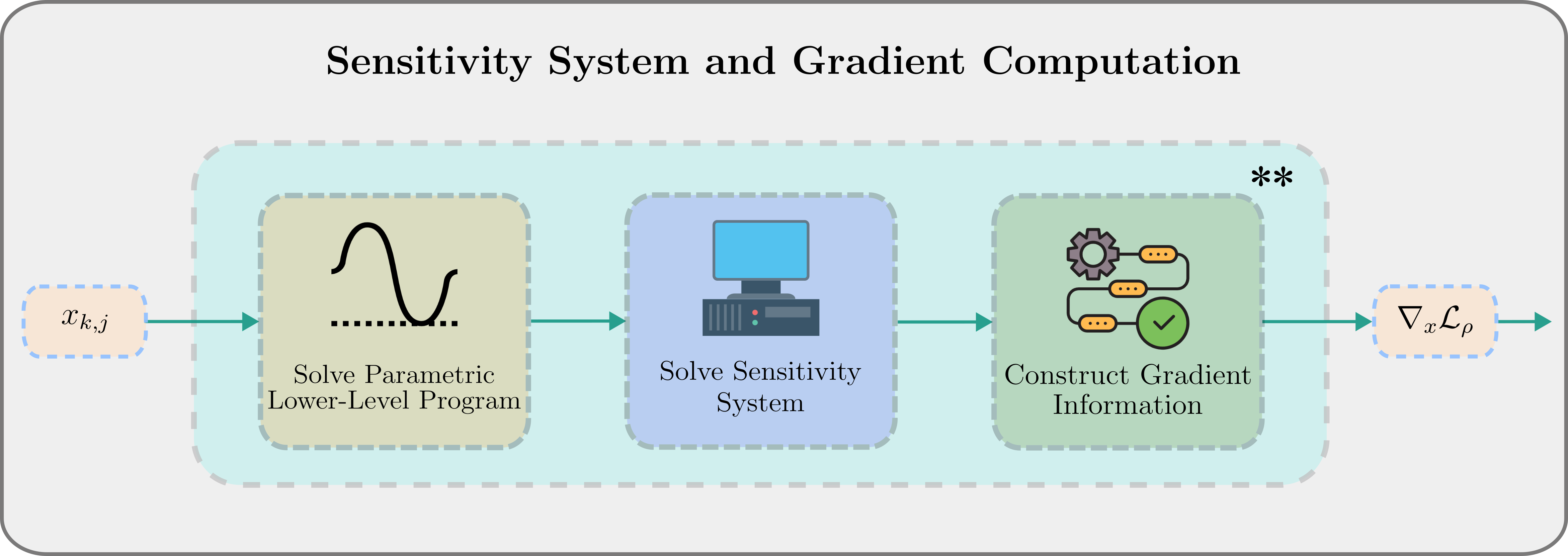}
	\caption{Detailed workflow for the implicit objective and gradient
		evaluation. This multi-step process, which includes solving
		the lower-level NLP and the adjoint linear system, is called
		at each iteration of the inner loop.}
	\label{fig:grad_computation}
\end{figure}

\begin{algorithm}[H]
	\caption{Upper-Level Update via Augmented Lagrangian Method}
	\label{alg:ALM_revised}
	\begin{algorithmic}[1]
		\State \textbf{Input:} Current iterates
		$(x_k, \bar{y}_k, \bar{\lambda}_k, \mu_k, \rho_k,
		r_{feas,k})$ and parameters $\gamma > 1$,
		$c \in (0,1)$, $\epsilon_{inner} > 0$.
		\State \textbf{Initialize inner loop:} Set $x_{k,0} \gets x_k$,
		$\bar{y}_{k,0} \gets \bar{y}_k$,
		$\bar{\lambda}_{k,0} \gets \bar{\lambda}_k$.
		\State Set $j \gets 0$.
		\Repeat
		\State Solve the adjoint system \eqref{eq:adjoint_system} with
		$q(x_{k,j},\hat{\mu}_k)$ defined in \eqref{eq:adjoint_rhs}
		and compute
		$\nabla_x\mathcal{L}_{\rho}(x_{k,j},\mu_k;\rho_k)$ using
		\eqref{eq:adjoint_gradient}, without explicitly forming
		$J = d\bar{y}/dx$.
		\State Compute descent direction $p_{k,j}$ using L-BFGS-B with
		gradient $\nabla_x\mathcal{L}_{\rho}(x_{k,j},\mu_k;\rho_k)$.
		\State Find step size $\alpha_{k,j}$ via a line search satisfying
		the strong Wolfe conditions \eqref{eq:strong_wolfe}.
		\State Update primal variables:
		$x_{k,j+1} \gets x_{k,j} + \alpha_{k,j}\,p_{k,j}$.
		\State Solve the lower-level problem \eqref{eq:ll} at $x_{k,j+1}$
		to get the KKT pair $(\bar{y}_{k,j+1},\bar{\lambda}_{k,j+1})$.
		\State $j \gets j+1$.
		\Until{$\|\nabla_x\mathcal{L}_{\rho}(x_{k,j},\mu_k;\rho_k)\|_{\infty}
			< \epsilon_{inner}$}
		\State \textbf{Set final iterates:} $x_{k+1} \gets x_{k,j}$,
		$\bar{y}_{k+1} \gets \bar{y}_{k,j}$, and
		$\bar{\lambda}_{k+1} \gets \bar{\lambda}_{k,j}$.
		\State Update dual variables:
		$\mu_{k+1,i} \gets
		\max\{0,\,\mu_{k,i} + \rho_k G_i(x_{k+1},\bar{y}_{k+1})\}$.
		\State Calculate constraint violation
		$r_{feas,k+1} \gets
		\|\max\{0, G(x_{k+1},\bar{y}_{k+1})\}\|_{\infty}$.
		\If{$r_{feas,k+1} > c\cdot r_{feas,k}$}
		\State $\rho_{k+1} \gets \gamma\cdot\rho_k$
		\Else
		\State $\rho_{k+1} \gets \rho_k$
		\EndIf \\
		\Return $(x_{k+1}, \bar{y}_{k+1}, \bar{\lambda}_{k+1},
		\mu_{k+1}, \rho_{k+1}, r_{feas,k+1})$.
	\end{algorithmic}
\end{algorithm}

\subsection{Convergence Analysis}

All results in this section are of a local nature. Convergence holds in
neighborhoods of $(\bar{x},\bar{y})$ where the lower-level KKT system
is strongly regular and the active set $A(x,\bar{y})$ remains constant.
On such neighborhoods the solution mapping $\bar{y}(x)$ and the reduced
functions $F(x,\bar{y}(x))$ and $G(x,\bar{y}(x))$ are continuously
differentiable, and the adjoint-based gradient \eqref{eq:adjoint_gradient}
is well defined. Globally, the reduced objective is only piecewise smooth
and may fail to be differentiable across active-set changes; accordingly,
the convergence and stationarity statements below apply locally on regions
of active-set stability.

For completeness, we provide a sketch of the convergence proof of
Algorithm~\ref{alg:MainAlg} using the upper-level update methodology
presented in Algorithm~\ref{alg:ALM_revised}. The proof adapts the
standard convergence analysis for the Augmented Lagrangian method
\citep{Nocedal2006} to our sensitivity-based framework for bilevel
optimization.

\begin{theorem}[Convergence to a KKT point]\label{thm:convergence}
	Let $\{x_k, \mu_k\}$ be a sequence of iterates generated by
	Algorithm~\ref{alg:MainAlg}, with the upper-level update step given
	by Algorithm~\ref{alg:ALM_revised}. Assume that:
	\begin{enumerate}
		\item[(a)] The regularity conditions in Assumption~\ref{ass:main}
		hold for every point in the sequence, and the inner solves
		in Algorithm~\ref{alg:ALM_revised} produce steps satisfying
		the strong Wolfe conditions \eqref{eq:strong_wolfe}.
		\item[(b)] The sequence of iterates $\{x_k\}$ is contained within
		a compact set $X$, and the corresponding sequence of generated
		multipliers $\{\mu_k\}$ is bounded.
	\end{enumerate}
	Then any limit point $(\bar{x}, \bar{\mu})$ of the sequence
	$\{x_k, \mu_k\}$ is a KKT point of \eqref{eq:Implicit_BLP}, with
	stationarity in the sense of \eqref{eq:adjoint_gradient}.
\end{theorem}

\begin{proof}
	First, we establish the existence of a limit point for the sequence
	of iterates. By Assumption~(b), the sequence $\{x_k\}$ is contained
	within a compact set $X$, and the sequence of multipliers $\{\mu_k\}$
	is bounded. This implies that the joint sequence $\{x_k, \mu_k\}$ is
	also contained within a compact set. Therefore, by the
	Bolzano--Weierstrass theorem, there exists at least one convergent
	subsequence. Let $(\bar{x}, \bar{\mu})$ be the limit point of such a
	subsequence, indexed by $\mathcal{K} \subseteq \mathbb{N}$, such that:
	\begin{equation}\label{eq:limit_point}
		\lim_{k \to \infty,\, k \in \mathcal{K}} (x_k, \mu_k)
		= (\bar{x}, \bar{\mu}).
	\end{equation}
	
	We now show that $(\bar{x}, \bar{\mu})$ satisfies the KKT conditions
	of the upper-level problem \eqref{eq:ul_KKT}.
	
	\textit{1. Dual feasibility:} The multiplier update rule in
	Algorithm~\ref{alg:ALM_revised}, given by \eqref{eq:mu_update},
	ensures that every component of $\mu_k$ is non-negative for all
	$k > 1$. Since the terms of the convergent subsequence are
	non-negative, their limit must also be non-negative, i.e.,
	$\bar{\mu} \geq 0$.
	
	\textit{2. Primal feasibility:} Consider the two possible behaviors
	of the penalty parameter sequence $\{\rho_k\}$.
	
	(a) The sequence remains bounded. This implies that there exists
	sufficiently large $\bar{k} \in \mathcal{K}$, after which the penalty
	parameter is no longer increased, i.e., $\rho_k = \rho_{\bar{k}}$
	for $k > \bar{k}$. By Algorithm~\ref{alg:ALM_revised}, this happens
	if the condition
	\begin{equation}
		\|\max\{0, G(x_{k+1}, \bar{y}_{k+1})\}\|_{\infty}
		\leq c \cdot \|\max\{0, G(x_k, \bar{y}_k)\}\|_{\infty},
	\end{equation}
	with $c \in (0,1)$, is satisfied for all $k > \bar{k}$. This implies
	that the sequence of feasibility residuals $\{r_{feas,k}\}$ converges
	to zero. Therefore, $r_{feas,k} \to 0$, implying
	$G(\bar{x}, \bar{y}) \leq 0$ by continuity.
	
	(b) The sequence diverges to infinity, i.e., $\rho_k \to \infty$.
	Suppose for contradiction that the limit point $\bar{x}$ is
	infeasible, meaning there is at least one constraint $j$ such that
	$G_j(\bar{x},\bar{y}) > 0$. By continuity of $G_j$ and $\bar{y}(x)$
	(Assumption~(a)), $G_j(x_k, \bar{y}_k) > 0$ for sufficiently large
	$k$. Since $\{x_k\}$ and $\{\mu_k\}$ are bounded, all other terms in
	$\mathcal{L}_\rho$ remain bounded. However, the penalty term for
	constraint $j$ grows asymptotically like
	$\frac{\rho_k}{2} G_j(x_k, \bar{y}_k) \to \infty$
	as $\rho_k \to \infty$.
	
	The strong Wolfe conditions \eqref{eq:strong_wolfe} ensure that the
	sequence of augmented Lagrangian values is non-increasing,
	\begin{equation}
		\mathcal{L}_{\rho}(x_{k+1}, \mu_k; \rho_k)
		\leq \mathcal{L}_{\rho}(x_k, \mu_k; \rho_k),
	\end{equation}
	and hence bounded above. This contradicts the divergence just
	established. Therefore, the initial assumption of infeasibility is
	false, and the limit point must be feasible.
	
	\textit{3. Stationarity:} By construction of
	Algorithm~\ref{alg:ALM_revised}, the primal update step satisfies
	the strong Wolfe conditions \eqref{eq:strong_wolfe}. Under Assumption 3.1, the reduced augmented Lagrangian is continuously differentiable on the neighborhood of active-set stability, and standard results for strong Wolfe line searches ensure the gradient norm converges to zero \citep{Nocedal2006}. Therefore:
	\begin{equation}
		\lim_{k \to \infty}
		\|\nabla_x \mathcal{L}_{\rho}(x_k, \mu_k; \rho_k)\| = 0.
	\end{equation}
	
	We substitute the dual variable update rule \eqref{eq:mu_update}
	into the adjoint gradient \eqref{eq:adjoint_gradient} and use
	continuity of the adjoint mapping (Assumption~(a)) to pass to the
	limit along $\mathcal{K}$. As $k \to \infty$ for $k \in \mathcal{K}$,
	we have $x_k \to \bar{x}$, $\mu_k \to \bar{\mu}$, and
	$\mu_{k+1} \to \bar{\mu}$. Taking the limit gives:
	\begin{subequations}
		\begin{align}
			0 &= \lim_{k \to \infty,\, k \in \mathcal{K}}
			\left\|\nabla_x F(x_k, \bar{y}_k)
			+ \mu_{k+1}^{\top} \nabla_x G(x_k, \bar{y}_k)\right\| \\
			&= \left\|\nabla_x F(\bar{x}, \bar{y})
			+ \bar{\mu}^{\top} \nabla_x G(\bar{x}, \bar{y})\right\| \\
			&= \|\nabla_x \mathcal{L}_F(\bar{x}, \bar{\mu})\|.
		\end{align}
	\end{subequations}
	This directly implies that the stationarity condition
	$\nabla_x \mathcal{L}_F(\bar{x}, \bar{\mu}) = 0$ is satisfied.
	
	\textit{4. Complementarity:} Having established primal feasibility,
	we consider the two cases for any constraint $j$ at the limit point.
	
	(a) The constraint is inactive, i.e., $G_j(\bar{x}, \bar{y}) < 0$.
	By continuity, there exists a $\bar{k}$ such that for all
	$k \in \mathcal{K}$ with $k > \bar{k}$ we have
	$G_j(x_k, y_k) < 0$. By Assumption~(b), the sequence $\{\mu_k\}$
	is bounded. Regardless of whether $\{\rho_k\}$ is bounded or
	diverges, the negative term $\rho_k G_j(x_k, y_k)$ is guaranteed to
	eventually dominate the bounded, non-negative $\mu_{k,j}$. This
	ensures that the term $\mu_{k,j} + \rho_k G_j(x_k, y_k)$ will be
	negative for all $k \in \mathcal{K}$ with $k > \bar{k}$, forcing
	$\mu_{k+1,j}$ to be zero via the $\max$ operator in
	\eqref{eq:mu_update}. Therefore, the limit $\bar{\mu}_j$ must be
	zero.
	
	(b) The constraint is active, i.e., $G_j(\bar{x}, \bar{y}) = 0$.
	In this case, the condition $\bar{\mu}_j G_j(\bar{x}, \bar{y}) = 0$
	is trivially satisfied.
	
	Since all KKT conditions are satisfied, any limit point
	$(\bar{x}, \bar{\mu})$ is a KKT point of the implicit problem
	\eqref{eq:Implicit_BLP}.
\end{proof}

\subsection{Equivalence to S-Stationarity}

We now establish the equivalence between a KKT point of the implicit
problem \eqref{eq:Implicit_BLP} and a stationary point of the MPCC
reformulation \eqref{eq:BLP MPCC}. A constraint qualification tailored
for these problems is \emph{MPEC-LICQ}, which requires that the
gradients of the active upper-level constraints, the active lower-level
constraints, and the lower-level stationarity equations are linearly
independent. This condition ensures that the multipliers of the MPCC
are well-defined \citep{Luo1996}.

The Lagrangian for the MPCC \eqref{eq:BLP MPCC} is given by:
\begin{equation}\label{eq:MPCC_Lagrangian}
	\mathcal{L}_{MPCC}(x,y,\lambda,\mu,\nu,\pi,\xi)
	= F(x,y) + \mu^{\top}G(x,y)
	+ \nu^{\top}\nabla_y \mathcal{L}_f(x,y,\lambda)
	+ \pi^{\top}g(x,y) - \xi^{\top}\lambda,
\end{equation}
where $\mu, \nu, \pi, \xi$ are the Lagrange multipliers following
standard sign convention: the equality constraint multiplier $\nu$ is
free, and the inequality constraint multipliers $\mu, \pi, \xi$ are
non-negative. The complementarity condition
$\bar{\lambda}_i g_i(\bar{x},\bar{y}) = 0$ for all $i$ is not
incorporated into \eqref{eq:MPCC_Lagrangian} directly; it is handled
through index-set dependent sign rules on the multipliers.

A feasible point $(\bar{x},\bar{y},\bar{\lambda})$ of the MPCC
\eqref{eq:BLP MPCC} is \emph{S--stationary} if there exist multipliers
$(\bar{\mu}, \bar{\nu}, \bar{\pi}, \bar{\xi})$ such that the following
hold:
\begin{enumerate}
	\item Stationarity:
	\begin{subequations}\label{eq:MPCC_stationarity}
		\begin{align}
			\nabla_x \mathcal{L}_{MPCC}
			&= \nabla_x F + \nabla_x G^{\top}\bar{\mu}
			+ (\nabla_{yx}^2 \mathcal{L}_f)^{\top}\bar{\nu}
			+ \nabla_x g^{\top}\bar{\pi} = 0,
			\label{eq:MPCC_stat_x}\\
			\nabla_y \mathcal{L}_{MPCC}
			&= \nabla_y F + \nabla_y G^{\top}\bar{\mu}
			+ (\nabla_y^2 \mathcal{L}_f)^{\top}\bar{\nu}
			+ \nabla_y g^{\top}\bar{\pi} = 0,
			\label{eq:MPCC_stat_y}\\
			\nabla_{\lambda} \mathcal{L}_{MPCC}
			&= (\nabla_y g)\,\bar{\nu} - \bar{\xi} = 0.
			\label{eq:MPCC_stat_lambda}
		\end{align}
	\end{subequations}
	
	\item Primal feasibility:
	$$G(\bar{x},\bar{y}) \leq 0,\quad
	\nabla_y \mathcal{L}_f(\bar{x},\bar{y},\bar{\lambda}) = 0,\quad
	g(\bar{x},\bar{y}) \leq 0,\quad
	\bar{\lambda} \geq 0,\quad
	\bar{\lambda}_i g_i(\bar{x},\bar{y}) = 0\ \forall i.$$
	
	\item Dual feasibility and complementarity slackness:
	$$\bar{\mu} \geq 0,\quad \bar{\pi} \geq 0,\quad \bar{\xi} \geq 0,$$
	$$\bar{\mu}^{\top}G(\bar{x},\bar{y}) = 0,\quad
	\bar{\pi}^{\top}g(\bar{x},\bar{y}) = 0,\quad
	\bar{\xi}^{\top}\bar{\lambda} = 0.$$
	
	\item Sign rules: Define index sets
	\begin{align*}
		\mathcal{I}^{+} &= \{i:\ g_i(\bar{x},\bar{y}) = 0,\
		\bar{\lambda}_i > 0\},\\
		\mathcal{I}^{-} &= \{i:\ g_i(\bar{x},\bar{y}) \leq 0,\
		\bar{\lambda}_i = 0\},\\
		\mathcal{I}^{0} &= \{i:\ g_i(\bar{x},\bar{y}) = 0,\
		\bar{\lambda}_i = 0\}.
	\end{align*}
	Then for every $i$:
	\begin{align*}
		\begin{cases}
			i \in \mathcal{I}^{+}: & \bar{\pi}_i \geq 0,\
			\bar{\xi}_i = 0,\\
			i \in \mathcal{I}^{-}: & \bar{\pi}_i = 0,\
			\bar{\xi}_i \geq 0,\\
			i \in \mathcal{I}^{0}: & \bar{\pi}_i \geq 0,\
			\bar{\xi}_i \geq 0.
		\end{cases}
	\end{align*}
\end{enumerate}

\begin{theorem}[Equivalence to S-stationarity]\label{thm:equivalence}
	Let $(\bar{x}, \bar{y})$ be a feasible point for the bilevel problem
	\eqref{eq:op_BLP}, and let $\bar{\lambda}$ be a multiplier such that
	$(\bar{x}, \bar{y}, \bar{\lambda})$ satisfies the KKT conditions of
	the lower-level problem \eqref{eq:ll}. Assume that the lower-level
	regularity conditions (Assumption~\ref{ass:main}) hold in a
	neighborhood of $(\bar{x}, \bar{y})$. Assume moreover that
	MPEC--LICQ holds for the MPCC \eqref{eq:BLP MPCC} at
	$(\bar{x}, \bar{y}, \bar{\lambda})$.
	
	Then, there exists $\bar{\mu}$ such that $(\bar{x}, \bar{\mu})$ is a
	KKT point of the implicit problem \eqref{eq:Implicit_BLP} if and only
	if $(\bar{x}, \bar{y}, \bar{\lambda})$ is an \mbox{S--stationary}
	point of the MPCC reformulation \eqref{eq:BLP MPCC}.
\end{theorem}

\begin{proof}
	The sensitivity matrix $M$ and sensitivity system are as defined in
	\eqref{eq:sensitivity_system}; by Assumption~\ref{ass:main}, $M$ is
	nonsingular. Denote the active and inactive index sets at
	$(\bar{x},\bar{y})$ by $A = \{i:\ g_i(\bar{x},\bar{y}) = 0,\
	\bar{\lambda}_i > 0\}$ and $I = \{i:\ g_i(\bar{x},\bar{y}) < 0,\
	\bar{\lambda}_i = 0\}$, respectively; SCC (Assumption~\ref{ass:main})
	rules out the biactive case.
	
	The stationarity condition of the implicit problem \eqref{eq:Implicit_BLP}
	at $(\bar{x},\bar{\mu})$, written via the chain rule, is:
	\begin{equation}\label{eq:UL-stationarity}
		\nabla_x F(\bar{x},\bar{y})
		+ \nabla_x G(\bar{x},\bar{y})^{\top}\bar{\mu}
		+ \left(\frac{d\bar{y}}{dx}\right)^{\top}
		\Bigl(\nabla_y F(\bar{x},\bar{y})
		+ \nabla_y G(\bar{x},\bar{y})^{\top}\bar{\mu}\Bigr) = 0.
	\end{equation}
	
	$(\Rightarrow)$ Suppose $(\bar{x},\bar{\mu})$ is a KKT point of
	\eqref{eq:Implicit_BLP}. The adjoint variables $(\bar{\nu},\bar{w})$
	are the unique solution of the adjoint system \eqref{eq:adjoint_system}
	evaluated at $(\bar{x},\bar{y},\bar{\lambda},\bar{\mu})$, with right-hand
	side given by $q(\bar{x},\bar{\mu})$ as defined in
	\eqref{eq:adjoint_rhs}. Applying the adjoint identity
	\eqref{eq:adjoint_gradient} to \eqref{eq:UL-stationarity} gives:
	\begin{equation}\label{eq:adjoint-equality}
		\left(\frac{d\bar{y}}{dx}\right)^{\top}
		\Bigl(\nabla_y F(\bar{x},\bar{y})
		+ \nabla_y G(\bar{x},\bar{y})^{\top}\bar{\mu}\Bigr)
		= \bigl(\nabla_{yx}^2 \mathcal{L}_f\bigr)^{\top}\bar{\nu}
		+ (\nabla_x g_A)^{\top}\bar{\Lambda}_A\bar{w},
	\end{equation}
	and substituting into \eqref{eq:UL-stationarity} yields:
	\begin{equation}\label{eq:MPCC_stat_x_adjoint}
		\nabla_x F(\bar{x},\bar{y})
		+ \nabla_x G(\bar{x},\bar{y})^{\top}\bar{\mu}
		+ \bigl(\nabla_{yx}^2 \mathcal{L}_f\bigr)^{\top}\bar{\nu}
		+ (\nabla_x g_A)^{\top}\bar{\Lambda}_A\bar{w} = 0.
	\end{equation}
	
	Define MPCC multipliers:
	\begin{equation}\label{eq:MPCC_multipliers}
		\bar{\pi}_I = 0,\quad
		\bar{\pi}_A = \bar{\Lambda}_A\bar{w},\quad
		\bar{\xi} = \nabla_y g(\bar{x},\bar{y})\,\bar{\nu}.
	\end{equation}
	Since $\bar{\Lambda}_A$ is diagonal with strictly positive entries,
	$\bar{\pi}_A \geq 0$ if and only if $\bar{w} \geq 0$. Then
	\eqref{eq:MPCC_stat_x_adjoint} is the MPCC stationarity condition
	with respect to $x$ \eqref{eq:MPCC_stat_x}.
	
	Stationarity with respect to $y$ follows from the first block of
	\eqref{eq:adjoint_system}:
	\begin{equation}
		H_{\mathcal{L}_f}\bar{\nu}
		+ \nabla_y g_A(\bar{x},\bar{y})^{\top}\Lambda_A\bar{w}
		= -\Bigl(\nabla_y F(\bar{x},\bar{y})
		+ \nabla_y G(\bar{x},\bar{y})^{\top}\bar{\mu}\Bigr),
	\end{equation}
	and since $H_{\mathcal{L}_f} = \nabla_{yy}^2 \mathcal{L}_f$ and
	$\bar{\pi}_I = 0$, this is \eqref{eq:MPCC_stat_y}. From the second
	block of \eqref{eq:adjoint_system}, $\nabla_y g_A\,\bar{\nu} = 0$,
	which implies $\bar{\xi}_A = 0$; hence \eqref{eq:MPCC_stat_lambda}
	is satisfied.
	
	Primal feasibility of $(\bar{x},\bar{y},\bar{\lambda})$ holds by
	construction, and complementarity $(\bar{\mu} \geq 0,\
	\bar{\mu}^{\top}G = 0)$ is inherited from the KKT conditions of the
	implicit problem.
	
	To verify the \mbox{S--stationarity} sign conditions: $\bar{\xi}_A = 0$
	follows from the adjoint system, and $\bar{\pi}_I = 0$ holds by
	construction. For $i \in A$, $\bar{\pi}_i$ represents the sensitivity
	of $F$ to the lower-level constraint $g_i$; since $\bar{x}$ is a local
	minimum, relaxing $g_i$ cannot improve $F$, so $\bar{\pi}_A \geq 0$.
	For $i \in I$, $\bar{\xi}_i$ is the sensitivity to $\lambda_i = 0$;
	a negative $\bar{\xi}_i$ would imply the objective improves by
	activating $g_i$, contradicting optimality of $\bar{x}$, so
	$\bar{\xi}_I \geq 0$. The \mbox{S--stationarity} sign rules are
	therefore satisfied:
	$$i \in A:\ \bar{\pi}_i \geq 0,\ \bar{\xi}_i = 0;\qquad
	i \in I:\ \bar{\pi}_i = 0,\ \bar{\xi}_i \geq 0.$$
	This proves that $(\bar{x},\bar{y},\bar{\lambda})$ is
	\mbox{S--stationary} for \eqref{eq:BLP MPCC}.
	
	$(\Leftarrow)$ Conversely, suppose $(\bar{x},\bar{y},\bar{\lambda})$
	is S--stationary for the MPCC, with multipliers
	$(\bar{\mu},\bar{\nu},\bar{\pi},\bar{\xi})$ satisfying
	\eqref{eq:MPCC_stationarity}. Under SCC (Assumption~\ref{ass:main}),
	$\bar{\Lambda}_A$ is invertible, and we define
	\begin{equation}\label{eq:w_construction}
		\bar{w} := \bar{\Lambda}_A^{-1}\bar{\pi}_A.
	\end{equation}
	Since $\bar{\pi}_A \geq 0$ and $\bar{\Lambda}_A$ has strictly positive
	diagonal entries, $\bar{w} \geq 0$, so the sign rules are preserved.
	The MPCC stationarity conditions with respect to $y$
	\eqref{eq:MPCC_stat_y} and with respect to $\lambda$
	\eqref{eq:MPCC_stat_lambda}, together with \eqref{eq:w_construction},
	are equivalent to the adjoint system \eqref{eq:adjoint_system}. The
	MPCC stationarity condition with respect to $x$ \eqref{eq:MPCC_stat_x},
	combined with the identity \eqref{eq:adjoint-equality}, which follows
	from the adjoint system, directly yields the implicit problem's
	stationarity condition \eqref{eq:UL-stationarity}. Feasibility and
	complementarity transfer directly. Hence $(\bar{x},\bar{\mu})$ is a
	KKT point of the implicit problem \eqref{eq:Implicit_BLP}.
	
	In summary, the adjoint variables $(\bar{\nu},\bar{w})$ provide the
	MPCC multipliers via $\bar{\pi}_A = \bar{\Lambda}_A\bar{w}$ and
	$\bar{\xi} = \nabla_y g\,\bar{\nu}$, while the upper-level multipliers
	$\bar{\mu}$ are preserved, establishing the equivalence between the
	reduced upper-level KKT points and S-stationary solutions of the MPCC
	reformulation.
\end{proof}

\section{Computational Tests}\label{Sect:Tests}

This section presents the computational validation of the proposed sensitivity-based algorithm, covering implementation details, illustrative examples, and systematic benchmark experiments.

\subsection{Implementation Details}
The method was implemented in Python, leveraging CasADi
\citep{Andersson2019} for its automatic differentiation capabilities.
All computations were performed on a Windows 10 Enterprise (64-bit)
workstation equipped with an Intel i5-6500 CPU (4 cores) and 16 GB RAM.
The lower-level parametric NLPs were solved using IPOPT
\citep{Wachter2006}. The inner ALM subproblem \eqref{eq:subproblem} was
solved using the L-BFGS-B algorithm \citep{Byrd1995} with a strong Wolfe
line search. Upper-level gradients were computed via the adjoint formula
\eqref{eq:adjoint_gradient}, avoiding explicit formation of the Jacobian
$J$.

The Augmented Lagrangian method was selected over a pure penalty
approach because it avoids the severe ill-conditioning associated with
large penalty parameters, and over sequential quadratic programming (SQP)
because the implicit gradient structure of the reduced problem makes
reliable second-order information expensive to obtain. The strong Wolfe
line search was preferred over a simple backtracking Armijo rule because
the curvature condition \eqref{eq:wolfe_curvature} is required to
guarantee that the L-BFGS-B Hessian approximation remains positive
definite throughout the inner iterations.

The tolerances were set to $\epsilon = 10^{-5}$ for the KKT residual
\eqref{eq:res_KKT}, $\epsilon_{\mathrm{inner}} = 10^{-6}$ for the inner
solve, and $\epsilon_{\mathrm{stall}} = 10^{-5}$ for the stall
criterion. The initial penalty parameter was set to $\rho_0 = 10$ and
the initial upper-level multipliers to $\mu_0 = 0$. For lower-level
problems with linear objectives, $\varepsilon$-regularisation with
$\varepsilon = 10^{-6}$ was applied to enforce uniqueness of the
lower-level solution. A well-known practical feature of the Augmented
Lagrangian method is that primal variables often converge much more
rapidly than the dual variables, whose first-order update can exhibit
slow linear convergence \citep{Nocedal2006}. Therefore, while the
primary termination criterion is the KKT residual falling below
$\epsilon$, a secondary stall criterion terminates the algorithm when
successive changes in $x$ and $F$ fall below $\epsilon_{\mathrm{stall}}$,
preventing excessive iterations spent refining dual variables once no
meaningful primal improvement remains.

At each inner iterate $x_{k,j}$, the lower-level problem \eqref{eq:ll}
is solved to obtain the KKT pair $(\bar{y}_{k,j}, \bar{\lambda}_{k,j})$
(Algorithm~\ref{alg:ALM_revised}, Step~9). Existence and local
uniqueness of this solution are guaranteed by Assumption~\ref{ass:main}:
the LICQ, SCC, and SSOSC conditions together ensure strong regularity of
the lower-level KKT system, which implies that the solution mapping
$\bar{y}(x)$ is locally single-valued and continuously differentiable in
a neighborhood of each iterate, so the lower-level solve is well-posed
throughout the algorithm.

For each benchmark problem, multiple starting points were selected by
manual exploration within the reported variable bounds, informed by
inspection of the upper-level objective landscape. The
best solution found across all initializations is reported in
Table~\ref{tab:results}. We note that systematic space-filling designs commonly used for multi-start strategies in single-level optimization — such as Latin hypercube sampling or uniform grids — are not directly applicable to bilevel problems: the bilevel feasible region is implicitly defined through the lower-level solution mapping $\bar{y}(x)$, and feasibility of a candidate point with respect to the upper-level constraints $G(x,\bar{y}(x)) \leq 0$, can only be assessed after solving the lower-level problem at $x_0$. Exploration of the implicit objective landscape therefore necessarily proceeds by solving the lower-level problem at each candidate point, making manual exploration guided by problem structure a natural and practical approach for the small-dimensional benchmark problems considered here.

\subsection{Test Problems}
As an illustrative example, we first conduct a detailed analysis of the
classic \texttt{Clark\-Westerberg\-1990} problem \citep{Clark1990}, a
well-known benchmark in the chemical engineering literature. The problem
is defined as:
\begin{equation}
	\begin{aligned}
		\min_x & \ (x - 3)^2 + (y - 2)^2\\
		\text{s.t.} & \ 0 \leq x \leq 8\\
		& \ y \in \Psi(x) = \argmin_{y}\set*{(y - 5)^2;
			\begin{aligned}
				&-2x +y-1 \leq 0,\\
				& x-2y+2 \leq 0, \\
				& x+2y-14 \leq 0
		\end{aligned}}.
	\end{aligned}
\end{equation}

The geometry of this problem is illustrated in
Figure~\ref{fig:CW1990_landscape}. The implicit upper-level objective,
$F(x,\bar{y}(x))$, is continuous but non-smooth, with non-differentiable
kinks at $x=2$ and $x=4$ that correspond to changes in the lower-level
active set. This non-convex landscape gives rise to multiple optima,
including a global minimum at $x=1$ and two distinct local minima. This
sensitivity to the initial point underscores the necessity of a
multi-start strategy to systematically explore the solution space, a
characteristic feature of non-convex bilevel problems.

\begin{figure}[t]
	\centering
	\includegraphics[width=1.0\textwidth]{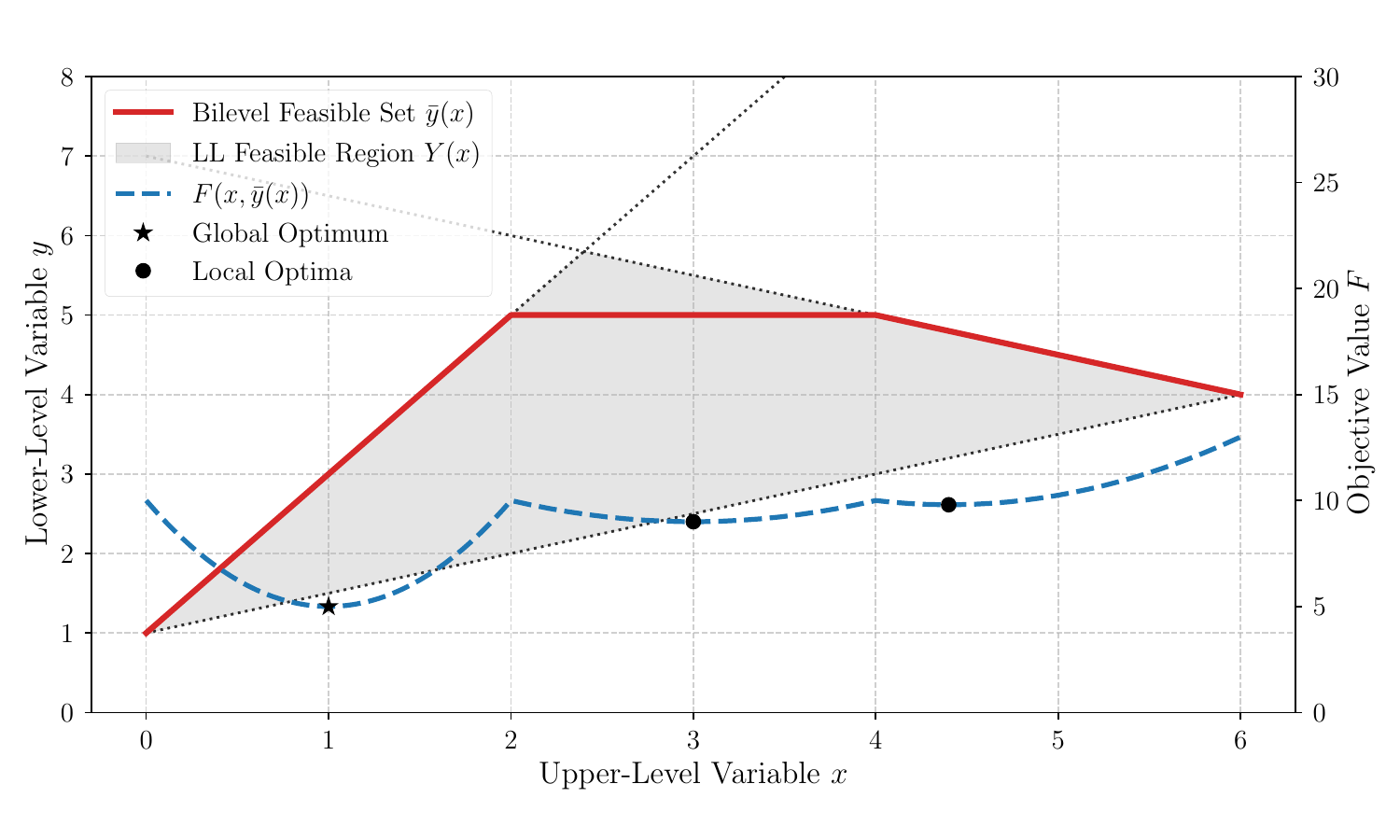}
	\caption{The implicit upper-level objective $F(x,\bar{y}(x))$ and
		the lower-level optimal response $\bar{y}(x)$ as a function
		of the upper-level variable $x$ for the
		\texttt{ClarkWesterberg1990} problem. The local and global
		optima are highlighted.}
	\label{fig:CW1990_landscape}
\end{figure}

To demonstrate the algorithm's performance on a problem with active
upper-level constraints, we present the convergence results for the
\texttt{Outrata\_Cervinka\_2009} problem in
Figure~\ref{fig:OC2009_plots}. This problem is defined as
\begin{equation}
	\begin{aligned}
		\min_x & \ -2x_1 - 0.5x_2 - y_2\\
		\text{s.t.} & \ x_1 \leq 0 \\
		& \ y \in \Psi(x) = \argmin_{y}\set*{y_1 - y_2 + x^\top y
			+ y^\top y;
			\begin{aligned}
				& y_2 - y_1 \leq 0,\\
				& y_2 + y_1 \leq 0, \\
				& y_2 \leq 0
		\end{aligned}}.
	\end{aligned}
\end{equation}
Unlike cases that terminate due to stalling, this problem demonstrates
convergence via the primary KKT criterion, as observed for several problems in Section~\ref{sect:discussion}. The plot shows the KKT
residual decreasing by several orders of magnitude to meet the
tolerance, while the upper-level multiplier $\mu_1$ converges rapidly
to its optimal value. This provides strong numerical evidence that the
algorithm performs as theoretically intended.

\begin{figure}[t]
	\centering
	\includegraphics[width=1.0\textwidth]{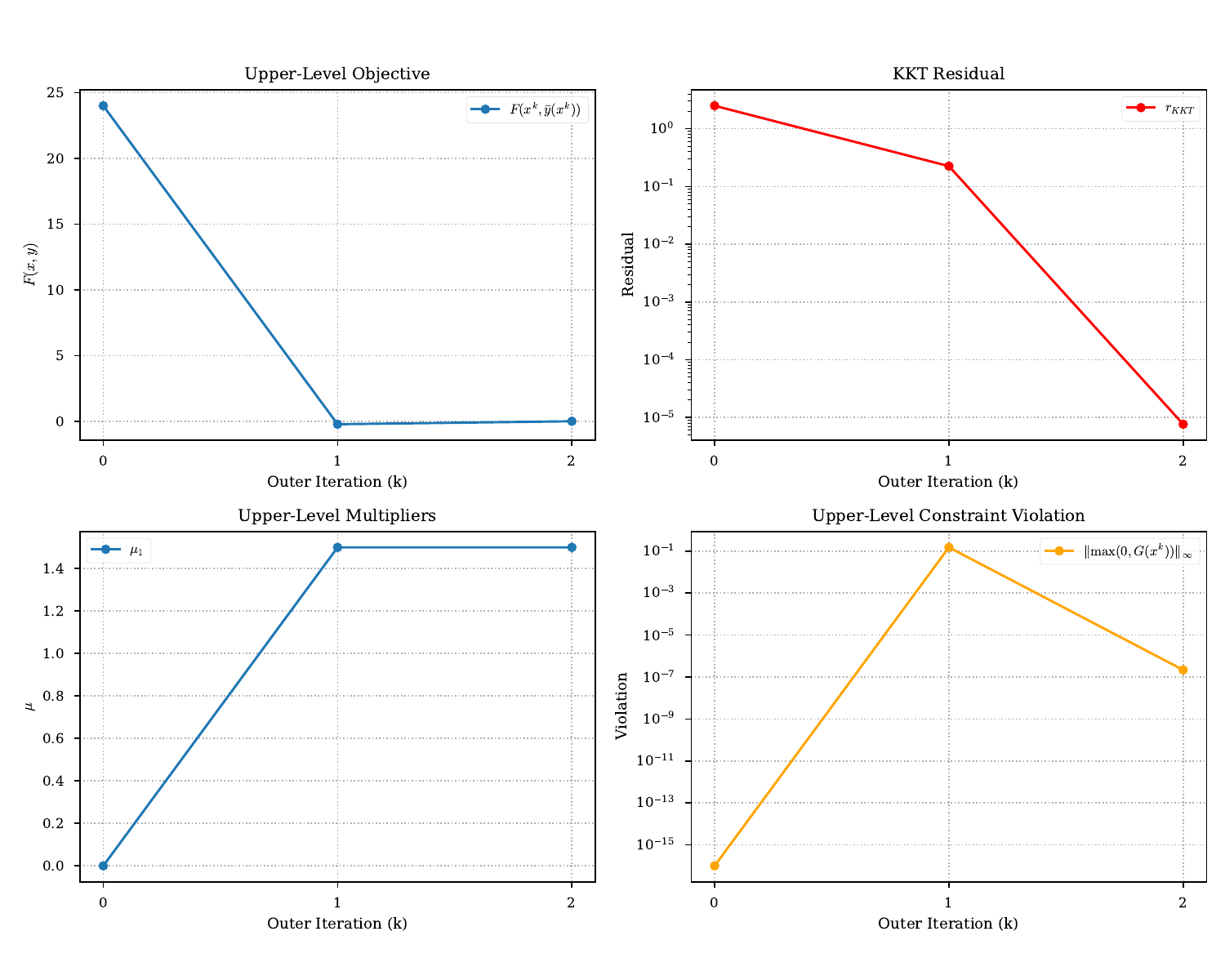}
	\caption{Convergence behavior for the \texttt{Outrata\_Cervinka\_2009}
		problem. The algorithm converges in a few iterations as the
		KKT residual drops below the tolerance $\epsilon = 10^{-5}$.}
	\label{fig:OC2009_plots}
\end{figure}

\subsection{Benchmark Results and Discussion}\label{sect:discussion}
To demonstrate the broader applicability and robustness of the proposed
method, the algorithm was tested on a suite of benchmark problems from
the BOLIB library \citep{Ward2025}. A summary of these computational
experiments is presented in Table~\ref{tab:results}. For each problem,
a multi-start strategy was employed, and the best solution found is
reported. The table details the problem dimensions as the tuple
$(n,m,r,s)$, where $n$ is the dimension of the upper-level decision
vector $x \in \mathbb{R}^n$, $m$ is the dimension of the lower-level
decision vector $y \in \mathbb{R}^m$, $r$ and $s$ are the number of
constraints in the upper and lower levels, respectively, including bound
constraints, as reported in the literature. It also reports the initial
point $x_0$ that led to the best solution, the known optimal value from
the literature $\bar{F}_r$, the value computed by the proposed method
$\bar{F}_c$, the number of outer ALM iterations, and the wall-clock
solution time in seconds.

\begin{table}[t]
\centering
\caption{Performance of the Sensitivity-Based ALM on Selected Benchmark Problems. Reported objective values are rounded to one decimal place; agreement with literature values is understood within the convergence tolerance $\epsilon = 10^{-5}$.}
\label{tab:results}
\rowcolors{2}{white}{gray!15}
\begin{tabular}{
		l
		c
		c
		S[table-format=3.1]  
		S[table-format=3.1]  
		S[table-format=2.0]  
		S[table-format=1.2]  
	}
	\toprule
	{Problem} & {($n,m,r,s$)} & {$x_0$} & {$\bar{F}_r$} & {$\bar{F}_c$} & {Iters} & {Time (s)} \\
	\midrule
	\texttt{AiyoshiShimizu1984Ex2} & $(2, 2,5,6)$ & $(20.0, 20.0)$ & 5.0 & 5.0 & 12 & 2.20 \\
	\texttt{AllendeStill2013} & $(2, 2,5,2)$ & $(2.0, 2.0)$ & -1.0 & -1.0 & 11 & 1.11 \\
	\texttt{Bard\_1988\_ex1} & $(1, 1,1,4)$ & $2.0$ & 17.0 & 17.0 & 11 & 0.91 \\
	\texttt{Bard\_1991\_ex1} & $(1, 2,2,3)$ & $4.0$ & 2.0 & 2.0 & 11 & 0.48 \\
	\texttt{Bard\_Book\_1998}$^\dagger$ & $(2, 2,4,7)$ & $(15, 15)$ & 0.0 & 0.0 & 6 & 0.28 \\
	\texttt{ClarkWesterberg1990} & $(1, 1,2,3)$ & $1.7$ & 5.0 & 5.0 & 11 & 0.33 \\
	\texttt{DempeEtal2012}$^\dagger$ & $(1, 1,2,2)$ & $0.9$ & -1.0 & -1.0 & 11 & 0.22 \\
	\texttt{Dempe\_Franke\_2011\_ex42}$^\dagger$ & $(2, 2,4,3)$ & $(-0.9, 0.9)$ & 3.0 & 3.0 & 11 & 2.75 \\
	\texttt{Dempe\_Lohse\_2011\_ex31a}$^\dagger$ & $(2, 2,0,4)$ & $(-0.4, -0.4)$ & -5.5 & -5.5 & 11 & 5.35 \\
	\texttt{Dempe\_Lohse\_2011\_ex31b}$^\dagger$ & $(3, 3,0,5)$ & $(4.0, 4.0, 4.0)$ & -12.0 & -12.0 & 12 & 1.13 \\
	\texttt{FloudasEtal2013} & $(2, 2,4,7)$ & $(10.0, 10.0)$ & 0.0 & 0.0 & 11 & 0.48 \\
	\texttt{Outrata\_Cervinka\_2009}$^\dagger$ & $(2, 2,1,3)$ & $(-10.0, -1.0)$ & 0.0 & 0.0 & 12 & 0.67 \\
	\texttt{Shimizu\_Aiyoshi\_1981\_ex2} & $(2, 2,3,4)$ & $(10.0, 1.0)$ & 225.0 & 225.0 & 39 & 4.20 \\
	\bottomrule
\end{tabular}
\parbox{\linewidth}{\footnotesize $^\dagger$ Lower-level objective 
	is linear; $\varepsilon$-regularisation with $\varepsilon = 10^{-6}$ 
	applied to enforce uniqueness of the lower-level solution.}
\end{table}

The results across the benchmark suite reveal several consistent
behavioral patterns. For problems with relatively simple landscapes,
the algorithm converges in very few outer iterations, reflecting the
efficiency of the quasi-Newton inner solver. Harder instances typically
exhibit rapid primal progress accompanied by slower dual convergence, a
well-known feature of the PHR multiplier update \eqref{eq:mu_update},
whose first-order scheme can exhibit slow linear convergence
\citep{Nocedal2006}. When constraints are strictly feasible at the
solution, the projection drives the corresponding multipliers to zero,
consistent with KKT conditions for inactive constraints. Near-binding
constraints, however, tend to produce small oscillations in the
multiplier trajectory and a slower decay of the complementarity
residual, making the stall criterion essential. This is illustrated by
\texttt{AiyoshiShimizu1984Ex2}, where the primal solution is identified
rapidly but convergence is declared via the stall criterion due to slow
dual refinement. The $\varepsilon$-regularisation stabilises lower-level
linear programme cases (marked by $^\dagger$ in Table~\ref{tab:results}) without altering solutions on regions where the
active set remains constant, as confirmed by the agreement between computed and literature values for these instances. Overall,
the results are consistent with the best-reported solutions in the
literature, demonstrating both the correctness and the robustness of the
proposed method.

Direct numerical comparison is not straightforward, as the sensitivity-based paradigm adopted here does not have a close algorithmic counterpart in the existing literature. Three alternative paradigms exist, each addressing different trade-offs between generality, computational cost, and solution quality.
Deterministic global methods based on bounding schemes, such as that of \cite{Mitsos2008} for continuous bilevel programmes with nonconvex lower levels, provide rigorous $\varepsilon$-optimality certificates. However, these methods require certified global optimality of the lower-level problem at every iteration, incurring substantial computational overhead that grows rapidly with problem dimension. For the smooth, locally well-posed problems considered here, this requirement is excessive given the analytic structure available, and such methods do not exploit the gradient information that this structure affords. Other deterministic approaches based on multiparametric programming \citep{Faisca2007,Avraamidou2019,Avraamidou2022} provide exact global solutions but are restricted to lower-level problems with linear or quadratic structure; they do not extend to the general smooth nonlinear lower levels addressed here. Furthermore, the explicit parametric solution map constructed by these methods scales combinatorially with the number of lower-level constraints and the dimension of the upper-level variable space, whereas the proposed method requires only a single lower-level NLP solve and one adjoint linear system per iteration, with per-iteration cost independent of the upper-level dimension $n$.

Data-driven methods such as DOMINO \citep{Beykal2020} and surrogate-assisted evolutionary algorithms \citep{Islam2017, Sinha2022} handle general nonlinear and black-box structures without requiring differentiability. However, these approaches construct a surrogate approximation of the upper-level objective from sampled lower-level solutions and optimize the surrogate in place of the original bilevel problem. The surrogate is not in general equivalent to the original bilevel structure: its optimum provides a feasible but only near-optimal solution, with no certificate on the gap to the true bilevel optimum \citep{Beykal2020}. Furthermore, surrogate accuracy degrades with problem dimension due to the curse of dimensionality, limiting scalability. In contrast, the proposed sensitivity-based method operates directly on the true bilevel problem without surrogate approximation, and exploits the available gradient structure therefore the computational effort is concentrated in solving the lower-level NLP and one adjoint linear system per gradient evaluation, both of which exploit the analytic structure of the problem rather than approximating it.

\section{Conclusions}\label{Sect:Conclusions}

In this work, a novel sensitivity-based algorithm for solving
continuous, optimistic bilevel optimization problems was developed. By
treating the lower-level problem as an implicit function of the
upper-level variables, this approach addresses the hierarchical
structure of BLPs, avoiding the need for classical KKT or
value-function reformulations. The method was embedded within a robust
Augmented Lagrangian framework, providing a theoretically sound and
practical tool for solving this challenging class of optimization
problems. Gradient evaluation via an adjoint system avoids explicit formation of the sensitivity Jacobian, reducing per-iteration cost independently of the upper-level dimension, and the computed KKT points are shown to be equivalent to S-stationary solutions of the associated MPCC reformulation under MPEC-LICQ.

Computational experiments on a suite of benchmark problems demonstrated
the effectiveness and efficiency of the proposed method. The results
highlighted the critical role of the architectural choice for the
inner-loop solver (Algorithm~\ref{alg:ALM_revised}); the use of a
robust quasi-Newton method (L-BFGS-B) proved essential for handling the
ill-conditioned and non-smooth subproblems that arise. The analysis of
the problem landscapes confirmed the non-convexity inherent in BLPs,
underscoring the necessity of a multi-start strategy. Furthermore, the
implemented dual-criterion stopping condition proved to be a robust and
efficient solution to the practical challenge of asymmetric convergence
rates between primal and dual variables in the Augmented Lagrangian
method.

Several directions remain open for future research. The most immediate is the extension to non-convex lower-level problems. Unlike global bounding approaches, which require certified global optimality of the lower level at every iteration, the sensitivity-based structure is amenable to extensions that admit locally optimal lower-level solutions — for instance through multi-start strategies at the lower level or branching on the lower-level feasible region — at the cost of weaker guarantees on the bilevel solution. A second direction concerns scalability: the benchmark suite used in this work consists of low-dimensional problems standard in the BLP literature, and the extension to larger-scale chemical engineering problems remains a primary avenue for future work, since the adjoint-based gradient structure offers the most significant computational advantage over reformulation-based approaches precisely in high upper-level dimensions. Additionally, exploring second-order update schemes for the dual variables could accelerate convergence and warrants further investigation. Finally, the multi-start procedure used here relies on manual exploration of the upper-level variable bounds; the development of a systematic initialization strategy adapted to the implicit structure of the bilevel feasible region remains an open direction.

\section*{CRediT authorship contribution statement}
\textbf{Eduardo Nolasco:} Writing – review \& editing, Writing – original draft, Visualization, Validation, Software, Methodology, Investigation, Formal analysis, Data curation, Conceptualization. \textbf{Ross D. King:} Supervision. \textbf{Vassilios S. Vassiliadis:} Conceptualization, Supervision.

\section*{Declaration of competing interest}
The authors declare that they have no known competing financial interests or personal relationships that could have appeared to influence the work reported in this paper.

\section*{Acknowledgments}
The first author acknowledges the CONACYT-SENER for financial support and Dr. Ehecatl Antonio del R\'{i}o Chanona for helpful discussions on improving the presentation of the manuscript.

\bibliographystyle{elsarticle-harv}
\bibliography{BLP_references} 
	
\end{document}